\documentclass[letterpaper,11pt]{amsart}
\usepackage[margin=1.2in]{geometry}
\usepackage{amsmath,amsthm,amssymb}
\usepackage{mathtools}
\usepackage{xspace,xcolor}
\usepackage[breaklinks,colorlinks,citecolor=teal,linkcolor=teal,urlcolor=teal,pagebackref,hyperindex]{hyperref}
\usepackage[alphabetic]{amsrefs}
\usepackage[all]{xy}
\usepackage[english]{babel}
\usepackage{enumitem}
\usepackage{tikz,xcolor}
\usepackage{tikz-cd}
\usepackage{mathrsfs}
\usepackage{color}

\theoremstyle{plain}
\newtheorem{thm}{Theorem}[section]

\newtheorem{lem}[thm]{Lemma}
\newtheorem{prop}[thm]{Proposition}
\newtheorem{cor}[thm]{Corollary}

\theoremstyle{definition}

\newtheorem{eg}[thm]{Example}

\theoremstyle{remark}
\newtheorem{rmk}[thm]{Remark}

\def\Mustata{Mus\-ta\-\c{t}\u{a}\xspace}

\def\N{{\mathbf N}}
\def\Z{{\mathbf Z}}
\def\Q{{\mathbf Q}}

\def\C{{\mathbf C}}

\def\A{{\mathbf A}}

\def\cB{\mathcal{B}}

\def\cD{\mathcal{D}}

\def\cH{\mathcal{H}}
\def\cI{\mathcal{I}}

\def\cK{\mathcal{K}}
\def\cL{\mathcal{L}}
\def\cM{\mathcal{M}}
\def\cN{\mathcal{N}}
\def\cO{\mathcal{O}}

\def\cQ{\mathcal{Q}}

\def\cX{\mathcal{X}}

\def\.{\cdot}
\def\^{\widehat}

\def\de{\partial}

\def\({\left(}
\def\){\right)}

\renewcommand{\and}{ \ \ \text{ and } \ \ }

\usepackage{microtype}

\begin{document}

\title[Microlocalization for LCI subvarieties]{Some applications of microlocalization for local complete intersection subvarieties}
\author[B.~Dirks]{Bradley Dirks}

\address{Department of Mathematics, Stony Brook University, Stony Brook, NY 11794-3651, USA}

\email{bradley.dirks@stonybrook.edu}

\thanks{The author was partially supported by NSF grant DMS-2001132 and NSF-MSPRF grant DMS-2303070.}


\maketitle

\begin{abstract} Saito's microlocalization construction has been used to great effect in understanding hypersurface singularities. In this paper, we introduce what we believe to be a suitable analogue of the microlocalization construction for local complete intersection subvarieties. As evidence, we relate our construction to Saito's in the codimension one case. 

Moreover, we use this construction to study various natural questions concerning the minimal exponent of LCI subvarieties. We show that the minimal exponent agrees with the smallest Bernstein-Sato root, which was expected to be true. We also show that, in the isolated complete intersection singularities case, the minimal exponent agrees with the smallest non-zero spectral number. As applications of these results, we prove constructibility of the function $x\mapsto \widetilde{\alpha}_x(Z)$ along certain Whitney stratifications and we prove that the spectrum (hence, the minimal exponent) is constant in equisingular families of ICIS varieties.
\end{abstract}

\section{Introduction} The study of hypersurface singularities has long been concerned with the theory of Bernstein-Sato polynomials $b_f(s)$. By work of Sabbah \cite{SabbahOrder}, these polynomials are intimately related to the $V$-filtration of Kashiwara and Malgrange on the $\cD$-module $\cB_f$ associated to $f$ (defined in Section \ref{sect-background}). For non-empty hypersurfaces, the Bernstein-Sato polynomial always has a trivial factor $(s+1)$, and dividing by that gives the \emph{reduced} Bernstein-Sato polynomial $\widetilde{b}_f(s)$.

By work of Lichtin \cite[Thm. 5]{Lichtin} and Koll\'{a}r \cite[Thm. 10.6]{KollarPairs}, the Bernstein-Sato polynomial is related to a well known singularity invariant: the log canonical threshold (for the definition and properties, see \cite[Ch. 9]{LazII}). In particular, one has \[{\rm lct}(f) = \min \{\gamma \mid b_f(-\gamma) = 0\}.\] In analogy with this equality, Saito \cite{SaitoMicrolocal} defined the \emph{minimal exponent} by a similar formula
\[ \widetilde{\alpha}(f) = \min \{\gamma \mid \widetilde{b}_f(-\gamma) = 0\}.\]

This invariant clearly contains more information than the log canonical threshold. In recent work, it has been related to Hodge ideals \cites{SaitoHodgeIdeal,MP2}, $k$-Du Bois \cites{MOPW,DBSaito} and $k$-rational singularities \cite{MPDB}*{Thm. E} (also in the Appendix by M. Saito in \cite{FL2}). The main approach to the study of the minimal exponent and the reduced Bernstein-Sato polynomial is the \emph{microlocal $V$-filtration}. Indeed, in the same way that the $V$-filtration is related to $b_f(s)$, Saito \cite{SaitoMicrolocal} defined a ``microlocal" analogue of the $V$-filtration which is related to $\widetilde{b}_f(s)$.

For subvarieties defined by several regular functions $f_1,\dots, f_r$, there is a notion of Bernstein-Sato polynomial \cite[Sect. 2.1]{BMS} and it also has a trivial root (in the LCI case), so one can consider its reduced version. There is also the $\cD$-module $\cB_f$ associated to $f_1,\dots, f_r$, which admits a $V$-filtration. Similarly to the hypersurface case, this $V$-filtration is related to the Bernstein-Sato polynomial. However, it is unclear how to define a ``microlocal" $V$-filtration for $r > 1$. The obvious generalization does not seem to be the correct answer. Similarly, in \cite{CDMO}, the author, along with Chen, \Mustata and Olano, defined the minimal exponent of an LCI subvariety, but it is not defined in analogy with the hypersurface case. Our definition makes use of the hypersurface $g = \sum_{i=1}^r y_i f_i$ on $X \times \A^r$. We use known properties of minimal exponents and $V$-filtrations of hypersurfaces to study this invariant. For details, see Section \ref{subsect-minexp}.

We define in this paper what we believe is the correct analogue of Saito's partial microlocalization construction for the case of LCI subvarieties. One caveat is that we are not able to define a filtration, but rather what could be viewed as the associated graded of a filtration. The construction uses the language and properties of Saito's theory of mixed Hodge modules. The necessary notions are reviewed in Section \ref{sect-background}, though the interested reader should consult the original papers \cite{SaitoMHM,SaitoMHP} and the survey \cite{Schnell}.

Following \cite[Sect. 4.3]{KashShap}, we define the \emph{microlocalization functor} $\mu(-) \coloneqq {\rm FL} \circ {\rm Sp}(-)$ which sends a mixed Hodge module on $X \times \A^r$ to a monodromic mixed Hodge module on $X \times \A^r$ (viewed as the conormal bundle of the zero section in the original $X\times \A^r$). See Section \ref{sect-background} for precise definitions of these functors. Using the microlocalization functor on $B_f^H$ (the Hodge module associated to $\cB_f$, see Section \ref{sect-background}), looking at the submodule supported on the zero section and by applying the inverse Fourier transform, we obtain a map
\[ \Psi \colon L \to {\rm Sp}(B_f^H)\]
where $L$ is isomorphic to the mixed Hodge module $i_* \Q_{Z\times \A^r}^H[\dim X]$ (see Section \ref{sect-microlocal}). We denote by $Q$ the cokernel of $\Psi$, and by $(\cQ,F)$ its underlying filtered $\cD$-module. Another way to view $Q$ is as the image of ${\rm Sp}(B_f^H)$ in $\cH^0 j_*j^*({\rm Sp}(B_f^H))$, where $j\colon U \to X \times \A^r$ is the inclusion of the complement of the zero section.

When $r = 1$, in Lemma \ref{comparer1} we relate $\cQ$ to the microlocalization construction of Saito. The map $\Psi$ and the module $Q$ are our analogues of the microlocalization construction.

As evidence that $\mu(\cB_f)$ and $\cQ$ are the correct generalizations of Saito's microlocalization functor to higher codimension, we use this construction to prove various results concerning the singularities of $Z$. For definitions of the terms used below, see Section \ref{sect-background} and the references therein.

First of all, we give in Proposition \ref{justification} an isomorphism of $\mu(\cB_f)$ with the vanishing cycles along the hypersurface $g = \sum_{i=1}^r y_i f_i$. This hypersurface is used in \cite{CDMO} to define the minimal exponent of the local complete intersection $Z$. This result illuminates the definition given in \cite{CDMO}. We review the properties of $\widetilde{\alpha}(Z)$ in Section \ref{sect-background}.

Secondly, we use the morphism $\Psi$ to study the spectrum of a (local) complete intersection $Z \subseteq X$, as defined by \cite{DMS} using the theory of mixed Hodge modules. This definition is a generalization of the definition by Steenbrink \cite{SteenbrinkSpec1} for isolated hypersurface singularities. The spectrum is 
\[ {\rm Sp}(Z,x) = \sum_{\alpha} \overline{m}_{\alpha,x} t^\alpha \in \Z[t^{1/e}],\]
for some $e\in \Z_{>0}$ and integers $\overline{m}_{\alpha,x}$. In the isolated hypersurface singularity case, it is widely known (see, for example, \cite{AGZV}) that ${\rm Sp}_{\rm min}(Z,x) \coloneqq\min \{ \alpha \mid \overline{m}_{\alpha,x} \neq 0\}$ is equal to the minimal exponent $\widetilde{\alpha}_x(Z)$. We prove an analogue of that statement. 
\begin{thm} \label{thm-Spectrum} Let $Z \subseteq X$ be a local complete intersection subvariety of pure codimension $r$ and let $x\in Z$. If $\widetilde{\alpha}_x(Z) > r - 1$, then
\[ \widetilde{\alpha}_x(Z) - r +1 \leq {\rm Sp}_{\rm min}(Z,x),\]
with equality if $Z_{\rm sing} = \{x\}$.
\end{thm}

\begin{rmk} Equality cannot hold outside of the isolated singularities case, even when $r =1$. Indeed, let $Z = \{f=0\} \subseteq X$ and assume $d = \dim Z_{\rm sing} > 0$, then we can take a suitable Whitney stratification $Z = \bigsqcup S_\alpha$ by \cite[Thm. 2(ii)]{DMS} so that if $x \in S$ lies in a dense stratum and $T \subseteq X$ is a smooth subvariety transverse to $S$ with $T\cap S =\{x\}$, we have, for example, by the discussion on \cite{DBSaito}*{Pg. 2}, the equality
\[ \widetilde{\alpha}_x(f) = \widetilde{\alpha}_x(f\vert_T),\]
but we also have the equality
\[{\rm Sp}(Z,x) = (-t)^{d} {\rm Sp}(Z\cap T,x),\]
and so by the isolated singularities case, we get
\[ \widetilde{\alpha}_x(f) = \widetilde{\alpha}_x(f \vert_T) ={\rm Sp}_{\rm min}(Z\cap T,x) = {\rm Sp}_{\rm min}(Z,x) - d.\]

It would be interesting to know if this difference is always an integer, as in this example. If that is true, it would also be interesting if the difference is related to some known invariant of stratified spaces.
\end{rmk}

By the work of Varchenko \cite{VarchenkoSemicont} and Steenbrink \cite{SteenbrinkSemicont}, we know that a $\mu$-constant family of isolated hypersurface singularities has constant spectrum, where $\mu$ is the Milnor number. It is then a natural question to ask if an ``equisingular" family of isolated complete intersection singularities has constant spectrum. To study this, we use the following set-up: let $X \subseteq Y \times T$ be a reduced complete intersection subvariety of pure codimension $r$ so that, for some distinguished point $0\in Y$, we have $\{0\} \times T \subseteq X$. Moreover, for all $t\in T$, we assume that $X_t = X \cap (Y \times \{t\})$ is a reduced complete intersection of codimension $r$ in $Y$ with isolated singular point at $0$. This is a \emph{reduced family of ICIS varieties}, where ICIS means \emph{isolated complete intersection singularity}. 

For hypersurfaces, there is a stronger notion of equisingularity in families which is to require the entire $\mu^*$-sequence (defined by Milnor numbers of general linear sections) be constant. By work of Teissier \cite[Chap. 3]{TeissierHypersurface} and Brian\c{c}on and Speder \cite{BrianconSpederWhitney}, this is equivalent to having $U = X \setminus (\{0\} \times T)$ be smooth and for the pair $(\{0\} \times T, U)$ to satisfy Whitney's conditions along $\{0\}\times T$. 

This latter condition (studied by Gaffney \cite{GaffneyICIS} when $Y$ and $T$ are affine spaces) immediately implies the following:
\begin{thm} \label{thm-equising} Let $X \subseteq Y \times T$ be a reduced family of ICIS varieties. Assume that $U$ is smooth and the pair $(\{0\} \times T, U)$ satisfies Whitney's conditions along $\{0\}\times T$. Assume moreover that $\cH^j {\rm DR}({\rm Sp}(\cB_f))\vert_{\{0\} \times T}$ is locally constant for all $j\in \Z$. Then ${\rm Sp}(X_t, (0,t))$ is independent of $t \in T$.
\end{thm}

Combining this with Theorem \ref{thm-Spectrum}, we obtain
\begin{cor} Under the assumptions of Theorem \ref{thm-equising}, the minimal exponent $\widetilde{\alpha}(X_t)$ is independent of $t\in T$.
\end{cor}

\begin{rmk} It would be interesting if the condition on the local constancy of $\cH^j{\rm DR}({\rm Sp}(\cB_f))$ when restricted to $\{0\} \times T$ were automatic in this situation, but we do not see an argument for that at the moment.

The equisingularity imposed in this theorem is much stronger than what is needed in the isolated hypersurface singularity case. In that case, one only needs that the Milnor number is constant to obtain constancy of the spectrum. It would be very interesting to argue that constant Milnor number for a family of ICIS subvarieties implies constant spectrum, or even just constant minimal exponent.
\end{rmk}

The final application concerns the roots of the Bernstein-Sato polynomial of $Z$. Let $Z$ be defined inside $X$ by $f_1,\dots, f_r$, which we assume form a regular sequence. The Bernstein-Sato polynomial $b_f(s)$ is a polynomial whose roots are intricately related to the singularities of $Z$, see \cite{BMS}. In the LCI case, the polynomial is divisible by $s+r$, so we can consider the \emph{reduced Bernstein-Sato polynomial} $\widetilde{b}_f(s) \coloneqq b_f(s)/(s+r)$. In analogy with the definition of the minimal exponent in the case $r = 1$, one would expect that $\widetilde{\alpha}(Z) = \min\{\gamma \mid \widetilde{b}_f(-\gamma) = 0\} = : \widetilde{\gamma}(Z)$. This was asked in \cite[Quest. 1.5]{CDMO}, and the relation between the two invariants was studied there. Here we give an affirmative answer, which by \cite[Thm. 1.3]{CDMO} also answers affirmatively \cite[Conj. 9.11]{MP3}:
\begin{thm} \label{thm-Bernstein} Let $Z$ be a local complete intersection of pure codimension $r$. Then
\[ \widetilde{\alpha}(Z) = \widetilde{\gamma}(Z).\]
\end{thm}

Finally, though this is not related to the microlocalization, we record the following fact concerning the constructibility of the roots of local Bernstein-Sato polynomials for arbitrary subvarieties $Z$. This is analogous to \cite[Rmk. 2.11]{SaitoMicrolocal}, and in the LCI case gives information about the minimal exponent.
\begin{thm} \label{thm-Construct} Let $Z \subseteq X$ be an arbitrary subvariety. Assume that $Z = \bigsqcup S_i$, $N_Z X= \bigsqcup S_i'$ are Whitney stratifications which stratify the projection $N_Z X \to Z$ and such that $\cH^j{\rm DR}({\rm Sp}(\cB_f))\vert_{S_i'}$ is locally constant for all $j\in \Z$. Let $m_\gamma(x) = \max \{ k \mid (s+\gamma)^k \mid b_{f,x}(s)\}$. Then $x \mapsto m_{\gamma}(x)$ is constructible with respect to $Z = \bigsqcup S_i$.

In particular, if $Z$ is LCI, then the function $x\mapsto \widetilde{\alpha}_x(Z)$ is constructible with respect to $\bigsqcup S_i$.
\end{thm}

\noindent {\bf Outline of Paper.} In Section \ref{sect-background}, we will mention briefly the important aspects of the theory of $\cD$-modules and mixed Hodge modules which are needed for the paper. We also recall the definition of the spectrum in the complete intersection case, due to \cite{DMS}. We end the background section with the definition of the Bernstein-Sato polynomial and minimal exponent for a local complete intersection.

In Section \ref{sect-microlocal}, we define the map $\Psi$ and study its properties. We relate this construction to the associated graded pieces of Saito's microlocalization construction when $r = 1$. Finally, we show how the composition of functors $\mu = {\rm FL} \circ {\rm Sp}$, classically called the microlocalization functor, relates to the definition of $\widetilde{\alpha}(Z)$ from \cite{CDMO}.

In Section \ref{sect-spectrum}, we prove Theorem \ref{thm-Spectrum} and Theorem \ref{thm-equising}. The main idea is to use the equality ${\rm Sp}(Z,x) = \widehat{\rm Sp}(Q,x)$, so that questions about the reduced spectrum can be studied via the spectrum of $Q$. This replacement is important for the application because the support of $Q$ lies over ${\rm Sing}(Z)$. Then, one observes that the minimal exponent allows one to obtain certain vanishings for the Hodge filtration of $Q$, which allow for vanishings of spectral numbers. 

In the final Section \ref{sect-Bernstein}, we prove Theorem \ref{thm-Bernstein} and Theorem \ref{thm-Construct}. Already in \cite[Sect. 6]{CDMO} the two quantities $\widetilde{\gamma}(Z)$ and $\widetilde{\alpha}(Z)$ had been related. The issue is to rule out the case that $\widetilde{\alpha}(Z)$ is strictly larger than $\widetilde{\gamma}(Z)$, when the latter is of the form $r + j$. This is done by relating the $G$-filtration on ${\rm Sp}(\cB_f)$ to that of $\cK[z_1,\dots, z_r]$, using the assumption that $\widetilde{\alpha}(Z) > r+j$. As before, it is not clear how to rule out these possible integer roots without using the microlocalization module $\cQ$ (equivalently, the module $\cK[z_1,\dots, z_r]$).

\noindent {\bf Acknowledgments}. The author is grateful to Qianyu Chen, James Hotchkiss, Lauren\c{t}iu Maxim, Mircea \Mustata, Sung Gi Park, Mihnea Popa and Claude Sabbah for useful conversations and comments about ideas appearing in this paper. The author would like to thank Sebasti\'{a}n Olano for many conversations around this topic which lead to the discovery of a gap in the proof of Theorem \ref{thm-Spectrum}.

\section{Background}\label{sect-background} \subsection{Notation} We assume throughout that $X$ is a smooth, irreducible complex algebraic variety of dimension $\dim X = n$. 

We will study various trivial vector bundles over $X$, so we introduce notation to keep track of a fixed coordinate system on the fibers. Starting from a trivial vector bundle $X\times \A^r_t$ with coordinates $t_1,\dots, t_r$ on the $\A^r$ factor, we will also study the normal bundle $X \times \A^r_z = N_{X \times \{0\}} X\times \A^r$, whose coordinates we take to be $z_1,\dots, z_r$ on the $\A^r$ factor. We will also consider the conormal bundle $X\times \A^r_y$, with coordinates $y_1,\dots, y_r$ which we view as ``dual" to $z_1,\dots, z_r$. 

 \subsection{$\cD$-modules} We begin this section with a brief review of $\cD$-modules, following \cite{HTT}.

On a smooth variety $S$, a $\cD_S$-module $\cM$ is a left module over the ring of differential operators $\cD_S$, as defined in \cite[Sect. 1.2]{HTT}. All $\cD$-modules considered here are holonomic, see \cite[Ch. 3]{HTT}. As in \cite[Sect. 1.5]{HTT}, given any morphism $f\colon S_1 \to S_2$ of smooth varieties, there are functors $f_+\colon{\rm D}_{\rm hol}^b(\cD_{S_1}) \to {\rm D}_{\rm hol}^b(\cD_{S_2})$ and $f^*\colon {\rm D}_{\rm hol}^b(\cD_{S_2}) \to {\rm D}_{\rm hol}^b(\cD_{S_1})$, where ${\rm D}^b_{\rm hol}(\cD)$ is the derived category of complexes with holonomic cohomology modules. The fact that these functors preserve holonomicity is \cite[Thm. 3.2.3]{HTT}. Moreover, any holonomic $\cD_S$-module $\cM$ has a constructible complex ${\rm DR}_S(\cM)$ associated to it, see \cite[Sect. 4.2]{HTT}. For $\cM$ holonomic, this complex is perverse, in the sense of \cite{BBDG}, by Kashiwara's constructibility theorem \cite[Thm. 4.6.6]{HTT}.

When $i\colon X \subseteq S$ is a closed embedding of a smooth subvariety in a smooth variety $S$, we know by \cite[Thm. 3.6]{CD} and \cite[Thm. 1]{CDS} that the pullback functor $i^*$ is computed using the \emph{Kashiwara-Malgrange $V$-filtration} of a (regular)-holonomic module $\cM$, which we briefly describe now.

We can consider a multiplicative filtration $V^\bullet \cD_S$ on the ring of differential operators on $S$ given by
\[ V^\bullet \cD_S = \{ P \in \cD_S \mid P \cdot \cI^k \subseteq \cI^{k+\bullet}\},\]
where $\cI$ is the coherent sheaf of ideals defining $X$ in $S$.

Kashiwara and Malgrange showed that, for any regular holonomic $\cD_S$-module $\cM$, there exists a $V$-filtration along $X$, which is indexed by $\Z$. Here we give the $\Q$-indexed definition, which is a refinement of the Kashiwara-Malgrange $V$-filtration, and is due to Saito. A $\Q$-indexed $V$-filtration only exists in the quasi-unipotent monodromy case, essentially by definition. This property is satisfied for all regular holonomic $\cD$-modules underlying mixed Hodge modules. 

A $\Q$-indexed $V$-filtration on $\cM$ is a filtration $V^\bullet \cM$ which is discretely indexed\footnote{This means that there is an increasing sequence $\alpha_j \in \Q$ with $\lim_{j\to -\infty} \alpha_j = -\infty$ and $\lim_{j\to \infty} \alpha_j = \infty$, such that $V^\alpha \cM$ for $\alpha \in (\alpha_j,\alpha_{j+1})$ only depends on $j$.}, left-continuously\footnote{This means $V^\alpha \cM = \bigcap_{\beta < \alpha}V^\beta$.}, so that if $\theta \in V^0\cD_S$ is any vector field which acts by the identity on $\cI/\cI^2$, the following properties hold:
\begin{enumerate} \item $V^i \cD_S \cdot V^\lambda \cM \subseteq V^{\lambda+i}\cM$ for all $i\in \Z, \lambda \in \Q$,
\item $V^\lambda \cM$ is coherent over $V^0 \cD_S$ for all $\lambda \in \Q$,
\item $V^1 \cD_S \cdot V^{\lambda}\cM = V^{\lambda+1}\cM$ for all $\lambda \gg 0$,
\item The operator $\theta - \lambda +r \in V^0\cD_S$ acts nilpotently on ${\rm Gr}_V^\lambda \cM \coloneqq V^\lambda \cM / V^{>\lambda} \cM$, where $V^{>\lambda} \cM = \bigcup_{\beta > \lambda}V^\beta \cM$.
\end{enumerate}

If $Z\subseteq S$ is a possibly singular subvariety, locally defined by $f_1,\dots, f_r \in \cO_S$, then for a $\cD_S$-module $\cM$, one defines the $V$-filtration along $Z$ to be the $V$-filtration of $\Gamma_+ \cM$ along the zero section $S \times \{0\} \subseteq S \times \A^r$. Here $\Gamma\colon S \to S \times \A^r$ is the graph embedding along $f_1,\dots, f_r$.

It is an exercise to show that if a $V$-filtration exists, then it is unique. From this, it is easy to see that if $\varphi\colon \cM \to \cN$ is a $\cD_S$-linear morphism of quasi-unipotent regular holonomic $\cD_S$-modules, then it is \emph{filtered strict} with respect to the $V$-filtration, meaning that for all $\lambda \in \Q$, we have equality
\[ \varphi(V^\lambda \cM) = V^\lambda \cN \cap \text{im}(\varphi).\]

The easiest example is the following:
\begin{eg} \label{KashiwaraEq} Let $\cM$ be a $\cD_S$-module which is supported on a smooth subvariety $X \subseteq S$. Assume that $X$ is defined by part of a system of local coordinates $t_1,\dots, t_c \in \cO_S$. Then Kashiwara's equivalence (\cite[Sect. 1.6]{HTT}) states that there is an isomorphism of $\cD_S$-modules
\[ \cM \cong \bigoplus_{\alpha \in \N^c} \cN \de_t^\alpha,\]
where $\cN = \bigcap_{i=1}^c \ker(t_i) \subseteq \cM$. It follows from the proof of Kashiwara's equivalence, or an easy exercise, that $\cN$ is also equal to $\ker(\sum_{i=1}^r \de_{t_i} t_i)$.

Then if $V^\bullet \cM$ is the $V$-filtration along $X$, it is easy to see that
\[ V^{-\lambda}\cM = \bigoplus_{|\alpha| \leq \lambda} \cN \de_t^\alpha,\]
and $V^{>0}\cM = 0$.
\end{eg}

Let $u \in \cO_S$ define a smooth hypersurface $H$. We define the \emph{nearby cycles} for $\lambda \in (0,1]$ of a regular holonomic $\cD_S$-module $\cM$ to be
\[ \psi_{u,\lambda}(\cM) : = {\rm Gr}_V^\lambda \cM,\]
and the \emph{vanishing cycles}
\[ \phi_{u,\lambda}(\cM)  = \psi_{u,\lambda}(\cM) \text{ for } \lambda \in (0,1), \quad \phi_{u,1}(\cM) = {\rm Gr}_V^0 \cM,\]
which are regular holonomic $\cD_H$-modules. These are so defined because under the Riemann-Hilbert correspondence, they correspond to the nearby and vanishing cycles of the perverse sheaf ${\rm DR}_S(\cM)$, see \cite{MMDuality}.

If $\varphi\colon \cM \to \cN$ is a morphism of regular holonomic $\cD_S$-modules such that $\varphi\vert_{S^*}$ is an isomorphism, where $S^* = \{u\neq 0\}$, then the result of Example \ref{KashiwaraEq} applied to the kernel and cokernel of the morphism $\varphi$ shows that for all $\lambda > 0$, the map $\varphi$ induces an isomorphism $\varphi\colon V^\lambda \cM \cong V^\lambda \cN$. In particular, if $\cM$ and $\cN$ are two regular holonomic $\cD_S$-modules with isomorphic restrictions to $S^*$, then
\begin{equation} \label{sameNearbyCycles} \psi_{u,\lambda}(\cM) \cong \psi_{u,\lambda}(\cN) \text{ for all } \lambda \in (0,1].\end{equation}

\subsubsection{Mixed Hodge Modules} We review here the pertinent parts of Saito's theory of mixed Hodge modules. For more details, one should see Saito's original papers \cites{SaitoMHP,SaitoMHM} or the survey article by Schnell \cite{Schnell}.

On a smooth variety $S$, the data of a mixed Hodge module consists of a tuple $(\cM,F,W,\cK,\alpha)$ where $\cM$ is a regular holonomic $\cD_S$-module, $F_\bullet \cM$ is a filtration of $\cM$ by $\cO$-submodules, called the Hodge filtration, which is a good filtration with respect to the order filtration on $\cD_S$, $W_\bullet \cM$ is a finite filtration by $\cD$-submodules, called the \emph{weight filtration}, and $\cK$ is a $\Q$-perverse sheaf with an isomorphism $\alpha: \cK\otimes_{\Q} \C \cong {\rm DR}_{S}^{\rm an}(\cM)$. In this paper, we will only be concerned with the filtered $\cD$-module $(\cM,F)$. Such a tuple is subject to various conditions concerning the Hodge filtration and the $V$-filtration along any locally defined (possibly singular) hypersurface in $S$. 

In this paper, following the conventions of \cite{DMS}, we will index the Hodge filtration following the conventions of \emph{right} $\cD$-modules. 
 
\begin{eg} \label{trivialHodgeModuleSmooth} On a smooth variety $S$, the structure sheaf underlies a Hodge module. Indeed, the tuple $(\cO_S ,F,W,\Q_{S^{\rm an}}[\dim S])$ is a Hodge module, where $F_{-\dim S}\cO_S = \cO_S$ and ${\rm gr}^W_{\dim S} \cO_S = \cO_S$. We denote this Hodge module by $\Q^H_S[\dim S]$.
\end{eg}

\begin{eg} \label{MHMonPoint} On a point $\{x\}$, the sheaf $\cD$ is isomorphic to the ring $\C$, and $\cD$-modules are simply $\C$-vector spaces.

By Saito \cite{SaitoMHM}*{(4.2.12)} The category of mixed Hodge modules on $\{x\}$ is equivalent to the category of graded polarizable $\Q$-mixed Hodge structures. There exists the ``trivial Hodge structure" $\Q^H$ whose underlying $\C$-vector space is $\C$, with Hodge filtration ${\rm Gr}^F_0 \C = \C$ and weight filtration ${\rm Gr}^W_0 \C = \C$.
\end{eg}

The category of mixed Hodge modules on a smooth variety is abelian. Mixed Hodge modules (and bifiltered direct summands thereof) satisfy many functoriality properties regarding maps between smooth algebraic varieties, which enhance those functoriality properties mentioned for holonomic $\cD$-modules. In this paper, we will only make use of the restriction functors along closed embeddings and the nearby/vanishing cycles functors. If $i\colon X \to S$ is a closed embedding and $M$ is a mixed Hodge module on $S$, then one can define
\[ i^* M \in {\rm D}^b({\rm MHM}(X))\]
whose underlying complex of $\cD_X$-modules agrees with $i^*(\cM)$, where $\cM$ is the $\cD_S$-module underlying $M$. By \cite[Thm. 1.2]{CD} and \cite[Thm. 1]{CDS}, this complex (along with its Hodge filtration) can be computed using the $V$-filtration of $\cM$ along $X$, but we will not use that in this paper. However, we will need to be careful about the indexing of the Hodge filtration for the cohomology modules of $i^* M$, which we state in the following remarks.

\begin{rmk} \label{RestrictHodge} Let $i \colon X \to S$ be a closed embedding of a smooth subvariety into a smooth variety. By a local computation (choosing coordinates such that a subset of them define $X$ in $S$), we see, using \cite[Thm. 1.2]{CD} and \cite[Thm. 1]{CDS} or iteratively \cite[Cor. 2.24]{SaitoMHM}, that if $(\cM,F)$ is a filtered $\cD_S$-module underlying a (direct summand of a) mixed Hodge module on $S$ which satisfies $F_\sigma \cM = 0$, then
\[ F_{\sigma - k} \cH^k i^* \cM = 0,\]
where $\cH^k i^* \cM \neq 0$ only for $k \in \{-r,\dots, 0\}$.
 
Moreover, if we factor $i = j \circ \iota\colon  X \to W \to S$, we have a natural isomorphism of functors $i^* \cong \iota^* \circ j^*$, and hence, we have a spectral sequence in ${\rm MHM}(X)$
\[ E_2^{p,q} = \cH^p \iota^* \cH^q j^* M \implies \cH^{p+q} i^* M.\]
\end{rmk}

\begin{rmk} \label{RestrictNonChar} If $(\cM,F)$ is a (direct summand) of a filtered $\cD_S$-module underlying a mixed Hodge module and $i\colon X \to S$ is a closed embedding of a smooth subvariety $X$ of pure codimension $c$, which is non-characteristic (in the sense of \cite[Lem. 2.25]{SaitoMHM}) with respect to $\cM$, then $i^* \cM = \cM \otimes_{\cO_S} \cO_X[c]$. Moreover, by \cite[Lem. 2.25]{SaitoMHM}, the Hodge filtration satisfies
\[ F_p i^* \cM = (F_{p-c} \cM) \otimes_{\cO_S} \cO_X [c].\]

Thus, if we can factor a closed embedding $i = j \circ \iota\colon X \to W \to S$ where $\iota$ is non-characteristic with respect to each cohomology module $\cH^p j^* M$ and $X$ has pure codimension $c$ in $W$, then using the spectral sequence in Remark \ref{RestrictHodge}, we get
\[ F_p \cH^k i^*\cM = (F_{p-c} \cH^{k-c} j^* \cM) \otimes_{\cO_W} \cO_X.\]
\end{rmk}

For any complex algebraic variety $Z$ with constant map $k\colon Z \to \{*\}$, we can define the \emph{trivial mixed Hodge module} on $Z$ by
\begin{equation} \label{trivialHodgeModule} \Q^H_Z \coloneqq k^* \Q^H \in {\rm D}^b({\rm MHM}(Z)).\end{equation}

The notation is consistent in the smooth case: if $S$ is smooth, then $\Q_S^H[\dim S]$ is the Hodge module described in Example \ref{trivialHodgeModuleSmooth}.

If $M$ is a mixed Hodge module on $S$ with underlying filtered $\cD_S$-module $(\cM,F)$ and $H \subseteq S$ is a smooth hypersurface defined by $u \in \cO_S(S)$, then the nearby cycles $\psi_u(M)$ is a mixed Hodge module on $H$. The underlying filtered $\cD_H$-module is
\[ \psi_u(\cM) = \bigoplus_{\lambda \in (0,1]} \psi_{u,\lambda}(\cM), \quad F_p \psi_u(\cM) = \bigoplus_{\lambda \in (0,1]} F_{p-1} {\rm Gr}_V^{\lambda}\cM.\]

Similarly, the unipotent vanishing cycles is a mixed Hodge module on $H$, with filtered $\cD_H$-module
\[ \phi_{u,1}(\cM) = {\rm Gr}_V^0 \cM, \quad F_p \phi_{u,1}(\cM) = F_p {\rm Gr}_V^0 \cM.\]

These shifts follow the convention for right $\cD$-modules.

\subsection{Specialization and Monodromic Mixed Hodge Modules} We describe a special class of mixed Hodge modules on a trivial vector bundle in this subsection. For more details, see \cite[Section 2.7]{CD}. Let $E = X\times \A^r_z$. A coherent $\cD_E$-module $\cM$ is \emph{monodromic} if it admits a decomposition
\begin{equation} \label{monoDecomp} \cM = \bigoplus_{\chi \in \Q} \cM^\chi,\end{equation}
where, if $\theta = \sum_{i=1}^r z_i \de_{z_i}$ is the Euler operator on $E$, we have $\cM^\chi = \bigcup_{\ell > 0} \ker\left( ( \theta - \chi + r)^\ell \right)$. It is easy to see that the $V$-filtration along $z_1,\dots, z_r$ is given by
\[ V^\lambda \cM = \bigoplus_{\chi \geq \lambda} \cM^\chi, \quad {\rm Gr}_V^\lambda \cM = \cM^\lambda.\] We say a mixed Hodge module $M$ on $E$ is \emph{monodromic} if its underlying $\cD_E$-module is monodromic. Note that our definition of monodromic assumes that the eigenvalues of $\theta$ are rational numbers, which is not standard in the literature, but if such a module underlies a mixed Hodge module, this condition is automatic.

For any $\lambda \in \Q$, we define $\cM^{\lambda+ \Z}$ as the $\cD$-module direct summand of $\cM$ given by $\bigoplus_{\ell \in \Z} \cM^{\lambda+\ell}$. If $(\cM,F)$ underlies a mixed Hodge module, we endow $\cM^{\lambda +\Z}$ with the induced filtration $F_\bullet \cM^{\lambda+\Z}$, so that, by Equation \ref{HodgeDecomp} below, it is a direct summand of a mixed Hodge module.

By \cite[Thm. 3.2]{CD}, for a monodromic $\cD_E$-module $\cM$, we know that the morphism
\begin{equation} \label{KoszulSurj} \bigoplus_{i=1}^r \cM^{\chi -1} \xrightarrow[]{z_i} \cM^\chi \end{equation}
is surjective for all $\chi > r$.

Moreover, if $(\cM,F)$ is a filtered monodromic $\cD$-module underlying a mixed Hodge module on $E$, then by \cite[Thm. 5.8]{CD}, we have
\begin{equation} \label{HodgeDecomp} F_p \cM = \bigoplus_{\chi \in \Q} F_p \cM^\chi, \quad \text{where} \quad F_p\cM^\chi = F_p\cM \cap \cM^\chi.\end{equation}
By \cite[Thm. 1.1]{CD} and \cite[Thm. 2]{CDS}, for all $\chi > r$ and $p\in \Z$, the complex
\begin{equation} \label{FilteredSurj}\bigoplus_{|I| = 2} F_p \cM^{\chi-2} \xrightarrow[]{z_i} \bigoplus_{i=1}^r F_p \cM^{\chi -1} \xrightarrow[]{z_i} F_p\cM^\chi  \to 0 \end{equation}
is exact, where the first term is the direct sum over all subsets $I \subseteq \{1,\dots, r\}$ of size 2. The first differential sends a tuple $(\eta_{\{i,j\}}) \in \bigoplus_{|I| =2} F_k \cM^{\chi-2}$ to the $r$-tuple $(\sum_{i > \ell} z_i \eta_{\{i,\ell\}} - \sum_{ j < \ell} z_j \eta_{\{j,\ell\}})_{1\leq \ell \leq r}$.

We have the following elementary lemma, which is an important ingredient in the proof of Theorem \ref{thm-Spectrum}.

\begin{lem} \label{lem-minHodgeGenSupp} Let $(\cM,F)$ be a filtered (direct summand of a) $\cD_E$-module underlying a monodromic mixed Hodge module, where $\cM = \bigoplus_{\ell\in \Z} \cM^{\lambda + \ell}$ for some $\lambda \in (0,1]$. Let $k = \min \{ p \in \Z \mid F_p \cM \neq 0\}$. If $F_k \cM^{r+\lambda-2} = 0$ and $F_k \cM^{r+\lambda-1} \neq 0$, then there does not exist $h \in \C[z_1,\dots, z_r]$ such that $h F_k \cM = 0$. In other words, the support of $F_k \cM$ as a $\C[z_1,\dots, z_r]$-module is equal to $\A^r_z$.
\end{lem}
\begin{proof} As $F_k \cM$ is a graded module, it suffices to assume $h$ is homogeneous. We will show that $h(\C[z]\cdot F_k \cM^{r+\lambda-1}) \neq 0$ for any homogeneous $h$ by induction on $\deg(h) = d$. In fact, we will prove the stronger claim by induction on $d$:
\begin{equation} \label{eq-propd} \text{ if } \sum_{|\alpha| = d} z^\alpha m_\alpha = 0 \text{ for } m_\alpha \in F_k \cM^{r+\lambda-1}, \text{ then } m_\alpha =0 \text{ for all } |\alpha| = d.\end{equation}

For $d = 1$, we look at the exact sequence \ref{FilteredSurj} when $\chi = r+\lambda+1$, giving
\[ \bigoplus_{|I| = 2} F_p \cM^{r+\lambda -2} \xrightarrow[]{z_i} \bigoplus_{i=1}^r F_p \cM^{r+\lambda-1} \xrightarrow[]{z_i} F_p\cM^{r+\lambda} \to 0.\]

By assumption, the first term is 0, so we have injectivity of $\bigoplus_{i=1}^r F_p \cM^{r+\lambda-1} \xrightarrow[]{z_i} F_p \cM^{r+\lambda}$, proving the claim when $d=1$.

Now, assume $\sum_{|\alpha| = d} z^\alpha m_\alpha = 0$. Define 
\[\eta_1 = \sum_{\alpha_1 > 0} z^{\alpha-e_1} m_\alpha, \quad \eta_2 = \sum_{\alpha_1 =0, \alpha_2 > 0} z^{\alpha-e_2} m_\alpha, \,\dots \, , \quad \eta_r = z_r^{d-1} m_{d e_r}.\] By the inductive hypothesis, it suffices to prove that $\eta_1 = \dots = \eta_r = 0$.

By assumption $(\eta_1,\dots, \eta_r)$ is an $r$-tuple which lies in the kernel of the map $\bigoplus_{i=1}^r F_k \cM^{r+\lambda+d-2} \xrightarrow[]{z_i} F_k \cM^{r+\lambda-1+d}$. By exactness of the complex \ref{FilteredSurj}, there exists $(\mu_{i,j}) \in \bigoplus_{|I| =2} F_k \cM^{r+\lambda+d-3}$ such that applying the Koszul differential gives $(\eta_1,\dots, \eta_r)$.

It is easy to see by iterated use of the surjectivity coming from the exact sequence \ref{FilteredSurj} that $F_k \cM^{r+\lambda+p-1} = \C[z_1,\dots,z_r]_p \cdot F_k \cM^{r+\lambda-1}$, where $\C[z_1,\dots, z_r]_p$ consists of the homogeneous polynomials of degree $p$. Write
\[ \mu_{\{i,j\}} = \sum_{|\gamma| = d-2}  z^\gamma m_{\gamma,\{i,j\}}.\]

Thus, the equality
\[ \eta_\ell = \sum_{i > \ell } z_i \mu_{\{i,\ell\}} - \sum_{ j < \ell } z_j \mu_{\{\ell,j\}}\]
is an equality in $\C[z_1,\dots, z_r]_{d-1}\cdot F_k \cM^{r+\lambda-1}$. In particular, by the inductive hypothesis, we see that the coefficient of $z^\beta$ on both sides must be equal, for any $|\beta| = d-1$.

We will prove by descending strong induction on $\ell$ that $\eta_\ell = 0$.

We have
\[ \eta_r = z_r^{d-1} m_{d e_r} = - \sum_{j < r} z_j \mu_{\{j,r\}},\]
and since the right hand side has coefficient 0 for $z_r^{d-1}$, we get $\eta_r = 0$.

Now, assume $\eta_{\ell+1} = \dots = \eta_r = 0$. We consider
\[ \eta_\ell = \sum_{i > \ell } z_i \mu_{\{\ell, i\}} - \sum_{ j < \ell } z_j \mu_{\{\ell,j\}}\]
where by definition, the only monomials appearing on the left hand side are $z^\beta$ with ${\rm supp}(\beta) = \{i\mid \beta_i > 0\} \subseteq \{\ell,\ell+1,\dots, r\}$. Fix some such $\beta$. Then the coefficient of $z^\beta$ on the right hand side is
\[ \sum_{i \in {\rm supp}(\beta)\setminus \{\ell\}} m_{\beta - e_i,\{\ell,i\}},\]
and we want to prove that this sum is zero.

To do this, let $j \in {\rm supp}(\beta) \setminus \{\ell\}$, so by choice of $\beta$ we must have $j > \ell$. Then
\[ 0 = \eta_j = \sum_{k > j } z_k \mu_{\{j,k\}} - \sum_{p < j} z_p \mu_{\{p,j\}}\]
and if we look at the coefficient of $z^{\beta + e_\ell - e_j}$, we get
\[ 0 = \sum_{k > j} m_{\beta + e_\ell - e_j -e_k, \{j,k\}} - \sum_{p < j} m_{\beta+e_\ell -e_j -e_p,\{p,j\}}.\]

By moving the term with $p = \ell$ to the other side ,we get
\[ m_{\beta-e_j,\{\ell,j\}} = \sum_{k > j} m_{\beta + e_\ell - e_j -e_k, \{j,k\}} - \sum_{\ell < p < j} m_{\beta+e_\ell -e_j -e_p,\{p,j\}}.\]

Adding up over all $j \in {\rm supp}(\beta) \setminus \{\ell\}$, we get zero. Indeed, for any pair $j < k$ (both in the support of $\beta$), the contribution from $m_{\beta-e_j,\{\ell,j\}}$ is $m_{\beta + e_\ell - e_j-e_k,\{j,k\}}$ and the contribution from $m_{\beta-e_k,\{\ell,k\}}$ is $-m_{\beta+e_\ell - e_j-e_k,\{j,k\}}$, so they cancel.
\end{proof}

Let $Z \subseteq X$ be defined by $f_1,\dots, f_r \in \cO_X(X)$ and let $M$ be a mixed Hodge module on $X$. Associated to this, we get a monodromic mixed Hodge module on $X\times \A^r_z$, called the \emph{Specialization of $M$}. By pushing forward along the graph embedding $\Gamma\colon X \to X\times \A^r_t$, we get a filtered $\cD$-module $\Gamma_+(\cM,F)$ on $X\times \A^r_t$ which we recall has Hodge filtration indexed following the conventions of right $\cD$-modules. Under the isomorphism $\Gamma_+ \cM = \bigoplus_{\alpha \in \N^r} \cM \de_t^\alpha$, this means
\[ F_p \Gamma_+ \cM = \bigoplus_{\alpha \in \N^r} F_{p-|\alpha|}\cM \de_t^\alpha.\]

We briefly recall the definition of the specialization functor ${\rm Sp}(-)$. For details, see \cite[Sect. 2]{BMS} or \cite[Sect. 2.4]{CD}.

Let $X \times \{0\} \subseteq X \times \A^r_t$ be the zero section. We can define the \emph{deformation to the normal bundle}
\[ \widetilde{X} \coloneqq \underline{\rm Spec}_{X\times \A^r_t}( \bigoplus_{k \in \Z} \cI^k \otimes u^{-k}),\]
where $\cI = (t_1,\dots, t_r)$ is the ideal defining the zero section and $\cI^{-\ell} = \cO_{X\times \A^r_t}$ for $\ell \geq 0$. The function $u\colon \widetilde{X} \to \A^1$ is smooth and over $\mathbf G_m \coloneqq \A^1 \setminus \{0\}$ is isomorphic to the projection $(X\times \A^r_t) \times \mathbf G_m \to \mathbf G_m$. Let $j\colon (X\times \A^r_t) \times \mathbf G_m \to \widetilde{X}$ be the open embedding, with complement $\iota\colon \{u =0\} \to \widetilde{X}$. 

The hypersurface $\{u =0\}$ is naturally isomorphic to the normal bundle of $X\times \{0\}$ inside $X \times \A^r_t$, which is again a trivial vector bundle $X\times \A^r_z$ over $X$, whose fiber coordinates are $z_1 = \frac{t_1}{u},\dots, z_r = \frac{t_r}{u}$.

If $q \colon (X\times \A^r_t) \times \mathbf G_m \to X\times \A^r_t$ is the projection to the first factor, then, by definition
\[ {\rm Sp}(\Gamma_* M) \coloneqq \psi_u j_* q^*(\Gamma_*M),\]
where for ease of notation we view $q^*(\Gamma_*M)$ as a mixed Hodge module placed in cohomological degree 0. 

The underlying $\cD$-module is monodromic, with decomposition
\[ {\rm Sp}(\Gamma_+ \cM) \coloneqq \bigoplus_{\chi \in \Q} {\rm Gr}_V^{\chi}(\Gamma_+ \cM).\]
The Hodge filtration is given by
\[ F_p {\rm Sp}(\Gamma_+ \cM) = \bigoplus_{\chi \in \Q} F_p {\rm Gr}_V^{\chi}(\Gamma_+ \cM).\]

Let $E^\vee = X \times \A^r_y$ be another trivial vector bundle, viewed as the dual of $E$. Given any $\cD_E$-module $\cM$, we can define a $\cD_{E^\vee}$-module ${\rm FL}(\cM)$, the \emph{Fourier-Laplace transform} of $\cM$, which has the same underlying $\cD_X$-module and such that, for any $m \in \cM$, we have
\[ y_i \cdot m = - \de_{z_i} \cdot m, \quad \de_{y_i} \cdot m = z_i \cdot m.\]

One can see that the functor ${\rm FL}(-)$ need not preserve regular holonomicity. However, Brylinski \cite{Brylinski} showed that if $\cM$ is a monodromic, regular holonomic $\cD_E$-module, then ${\rm FL}(\cM)$ is a monodromic, regular holonomic $\cD_{E^\vee}$-module.

The Fourier-Laplace transform of a monodromic mixed Hodge module is studied in \cite[Sect. 7]{CD}. In \emph{loc. cit.}, the Hodge filtration is shown to be related to that of $\cM$ in the following way:
\begin{equation} \label{HodgeFL} F_p {\rm FL}(\cM)^{r-\chi} = F_{p-\lceil \chi \rceil} \cM^\chi. \end{equation}

See also \cite{MonoMHM2}, where the same formula is shown, and is extended to irregular Hodge filtrations.

\subsection{Spectrum of a complete intersection subvariety} We give here the definition of the spectrum of a complete intersection $Z \subseteq X$ at a point $x\in Z$, as given by \cite{DMS}. The complete intersection assumption simplifies the definition of the spectrum.

Let $f_1,\dots, f_r \in \cO_X(X)$ define the complete intersection subvariety $Z \subseteq X$ of pure codimension $r$. Then the normal bundle $N_Z X$ is a trivial bundle over $Z$, and it embeds into $X \times \A^r_z$. For $x\in Z$ fixed, the fiber of $N_Z X \to Z$ over $x$ is then isomorphic to $\A^r$. Let $i_x\colon \{x\} \times \A^r_z \to X \times \A^r_z$ and for $\xi \in \{x\} \times \A^r_z$, let $j_\xi\colon \{\xi\} \to \{x\} \times \A^r_z$ be the inclusion. Set $i_\xi \coloneqq i_x \circ j_\xi\colon \{\xi\} \to X \times \A^r_z$.

For $\xi$ sufficiently general, i.e., so that $j_\xi$ is non-characteristic with respect to $\cH^k i_x^* {\rm Sp}(\cB_f)$ for all $k \in \Z$, we define
\[ \widehat{\rm Sp}(Z,x) \coloneqq \sum m_{\alpha,x} t^\alpha\]
where
\[ m_{\alpha,x} \coloneqq \sum_{k \in \Z} (-1)^{k} \dim_\C {\rm Gr}^F_{\lceil \alpha \rceil - \dim Z -1} \cH^{k-r} i_\xi^* {\rm Sp}(\cB_f)^{\alpha + \Z}.\]

The \emph{reduced spectrum} is defined as
\[ {\rm Sp}(Z,x) \coloneqq \widehat{\rm Sp}(Z,x) + (-t)^{\dim Z +1} = \sum \overline{m}_{\alpha,x} t^\alpha.\]

This definition can easily be extended to an arbitrary monodromic mixed Hodge module $M$ on $X\times \A^r_z$: let
\[ \widehat{\rm Sp}(M,x) \coloneqq \sum m_{\alpha,x}(M) t^\alpha,\]
where
\[ m_{\alpha,x}(M) \coloneqq \sum_{k\in \Z} (-1)^k \dim_\C {\rm Gr}^F_{\lceil\alpha \rceil - \dim Z -1} \cH^{k-r} i_\xi^* \cM^{\alpha + \Z},\]
where $(\cM,F)$ is the underlying filtered $\cD_E$-module.

If $0 \to M_1 \to M_2 \to M_3 \to 0$ is a short exact sequence of monodromic mixed Hodge modules on $E$, it is easy to check that, for all $x\in Z$, we have
\[ \widehat{\rm Sp}(M_2,x) = \widehat{\rm Sp}(M_1,x) + \widehat{\rm Sp}(M_3,x).\]

We state here an easy lemma that allows us to give vanishing of spectral numbers.

\begin{lem} \label{lem-SpectralNumberLowerBound} Let $(\cM,F)$ be a filtered $\cD_E$-module underlying a monodromic mixed Hodge module $M$. Let $p \in \Z$ and $\lambda \in (0,1]$ be such that $F_{p-1-\dim X} \cM = 0$ and $F_{p-\dim X}\cM^{\mu+\Z} = 0$ for all $\mu \in (0,\lambda)$. Then
\[ m_{\alpha,x}(M) = 0 \text{ for all } \alpha < p + \lambda.\]
\end{lem}
\begin{proof} Remark \ref{RestrictHodge} shows that $F_{p-1-\dim X} \cH^j i_x^* \cM = 0$ and $F_{p-\dim X} \cH^j i_x^* \cM^{\mu+\Z} = 0$ for all $\mu \in (0,\lambda)$.

As $j_\xi$ is non-characteristic, we have by Remark \ref{RestrictNonChar} that 
\[F_{p-1-\dim Z} \cH^{j-r} i_\xi^* \cM = 0 \text{ and } F_{p-\dim Z} \cH^{j-r} i_\xi^* \cM^{\mu+\Z} = 0\] for all $\mu \in (0,\lambda)$.

Thus, we get $m_{\alpha,x}(M) = 0$ for all $\lceil \alpha\rceil - 1 - \dim Z \leq p - 1 - \dim Z$, i.e., for all $\alpha \leq p$. Moreover, for any $\alpha$ with $\alpha - \lfloor \alpha \rfloor = \mu \in (0,\lambda)$, we get $m_{\alpha,x}(M) = 0$ for all $\lceil \alpha \rceil -1 - \dim Z \leq p - \dim Z$. This gives the vanishing for all $\alpha < p+\lambda$, as desired.
\end{proof}

\subsection{Bernstein-Sato polynomials} An important invariant of singularities for a hypersurface is the \emph{Bernstein-Sato polynomial}, which we will define through use of the $G$-filtration, following the notation of \cite{SaitoMicrolocal}. Let $f_1,\dots, f_r \in \cO_X(X)$ define a subvariety $Z$, not assumed to be a complete intersection, or even reduced or irreducible. Let $\cB_f = \bigoplus_{\alpha \in \N^r} \cO_X \de_t^\alpha \delta_f$ be the $\cD$-module obtained by pushing $\cO_X$ forward along the graph embedding $\Gamma\colon X \to X\times \A^r_t$. 

The module $\cB_f$ admits an exhaustive, $\Z$-indexed filtration $G^\bullet \cB_f$ defined by
\[ G^\bullet \cB_f \coloneqq V^\bullet \cD_{X\times \A^r_t} \cdot \delta_f.\]

We define the \emph{Bernstein-Sato polynomial} $b_f(s)$ to be the minimal polynomial of the action of $s = - \sum_{i=1}^r \de_{t_i} t_i$ on the associated graded module ${\rm gr}_G^0 \cB_f$. By work of Kashiwara, if $b_f(\gamma) = 0$ then $\gamma \in \Q_{<0}$.

A basic exercise in linear algebra shows that the following holds, by definition of $b_f(s)$:
\begin{equation} \label{GFiltRoots} (s+\gamma)^k \mid b_f(s) \iff (s+\gamma)^{k-1} {\rm Gr}_V^{\gamma} {\rm gr}_G^0 \cB_f \neq 0.\end{equation}

If $Z$ is a reduced complete intersection subvariety of $X$ of codimension $r$, then by restricting to $V \subseteq X$ such that $Z \cap V$ is smooth and dense in $Z$, we see that $s+r$ divides $b_f(s)$. Hence, we can define the \emph{reduced Bernstein-Sato polynomial} by the quotient $\widetilde{b}_f(s) \coloneqq b_f(s)/(s+r)$. If $r =1$, meaning $Z$ is a hypersurface, then Saito \cite{SaitoMicrolocal} defines the minimal exponent $\widetilde{\alpha}(f)$ to be the smallest root of $\widetilde{b}_f(-s)$.

\subsection{Definition and Properties of $\widetilde{\alpha}(Z)$} \label{subsect-minexp} In \cite{CDMO}, for $Z \subseteq X$ a local complete intersection subvariety, the minimal exponent $\widetilde{\alpha}(Z)$ is defined through the following construction: let $Y = X \times \A^r_y$. Assume that $Z$ is defined by $f_1,\dots, f_r \in \cO_X(X)$ and consider the function $g = \sum_{i=1}^r y_i f_i$ on $Y$. Let $U = Y \setminus (X\times \{0\})$ be the complement of the zero section in $Y$. We define
\[ \widetilde{\alpha}(Z) \coloneqq \widetilde{\alpha}(g \vert_U).\]

It is not hard to see that this agrees with the usual definition if $Z$ is a hypersurface. 

For $x\in Z$, the minimal exponent of $Z$ at $x$ is defined \cite[Defn. 4.16]{CDMO} to be $\widetilde{\alpha}_x(Z) \coloneqq \max_{x\in V} \widetilde{\alpha}(V,V\cap Z)$, where $x\in V$ is an open neighborhood of $x$ in $X$.

\begin{rmk} We review here some results concerning the minimal exponent for LCI varieties. Let $Z \subseteq X$ be a complete intersection subvariety of the smooth connected variety $X$ of pure codimension $r$ defined by $f_1,\dots, f_r \in \cO_X(X)$.

\begin{enumerate} \item (\cite[Rmk. 4.8]{CDMO}) The invariant $\widetilde{\alpha}(Z)$ does not depend on the choice of $r$ generators of the ideal defining $Z$. Hence, one can define the minimal exponent for subvarieties which are LCI but not necessarily complete intersection. 

\item (\cite[Prop. 4.14]{CDMO}) The difference $\widetilde{\alpha}(Z) - r$ depends only on $Z$ and not on the embedding $Z \subseteq X$.

\item (\cite[Rmk. 4.5]{CDMO}) Let ${\rm lct}(X,Z)$ be the log canonical threshold, then
\[ {\rm lct}(X,Z) = \min\{\widetilde{\alpha}(Z),r\}.\]

\item (\cite[Thm. 1.3]{CDMO}, \cite[Thm. 1.1]{CDM}) The variety $Z$ has at worst $k$-Du Bois singularities if and only if $\widetilde{\alpha}(Z) \geq r+k$. It has at worst $k$-rational singularities if and only if $\widetilde{\alpha}(Z) > r+k$. See \cite[Sect. 13]{MP3}, \cite[Sect. 2.4]{CDM} (resp. \cite[Sect. 2.5]{CDM}) for the definition of $k$-Du Bois singularities (resp. $k$-rational singularities).
\end{enumerate}

\end{rmk}

\section{Partial Microlocalization} \label{sect-microlocal} In this section, we define the map $\Psi$ from the introduction and describe its properties. We compare this construction to Saito's partial microlocalization in the case $r=1$ and also relate $\mu(B_f^H)$ to the vanishing cycles of the hypersurface $g$.

The map $\Psi$ is most easily understood on the dual vector bundle. Let $\mu(B_f^H) = {\rm FL}{\rm Sp}(B_f^H)$ be the microlocalization of the mixed Hodge module $B^H_f$. This is a monodromic mixed Hodge module on $E^\vee$, the dual bundle to $E$. Recall that we have fiber coordinates $y_1,\dots, y_r$ chosen on $E^\vee$. Then we can consider the submodule of $\mu(B_f^H)$ of sections with support on $X \times \{0\} = V(y_1,\dots, y_r) \subseteq E^\vee$. This leads to an injective morphism of mixed Hodge modules
\[ {\rm FL}(\Psi) \colon \Gamma_{X\times \{0\}}(\mu(B_f^H)) \to \mu(B_f^H).\]

By applying the inverse Fourier transform $\overline{\rm FL}$ \cite[Cor. 1.6]{CD} to ${\rm FL}(\Psi)$, we get a map of mixed Hodge modules on $E$,
\[ \Psi\colon L \to {\rm Sp}(B_f^H),\]
which is injective (as the map on the level of $\cD_X$-modules has not changed). We let $Q \coloneqq \text{coker}(\Psi)$, with underlying $\cD_E$-module $\cQ$.

\begin{rmk} As pointed out to the author by Sabbah, since the microlocalization is a \emph{localization}, one should be careful to use the \emph{derived} version of this construction, in general. As mentioned in the introduction, ${\rm FL}(Q)$ can be thought of as the image of the natural map $\mu(B_f^H) \to \cH^0 j_* j^*(\mu(B_f^H))$, where $j$ is the open embedding of the complement of the zero section into $E^{\vee}$. 

In general, it is possible that $j_* j^*(\mu(B^f_H))$ has non-trivial higher cohomology, however, this is not the case for an LCI subvariety, which is all that we consider in this paper. Indeed, the higher cohomology of $j_*j^*(\mu(B^f_H))$ is isomorphic to the higher cohomology of $\sigma_* \sigma^! \mu(B_f^H)$, where $\sigma \colon X \times \{0\} \to E^\vee$ is the inclusion of the zero section. But it is immediate from \cite[Thm. 1.2]{CD} and \cite[Thm. 1]{CDS} that $\sigma_* \sigma^! \circ {\rm FL} = {\rm FL} \circ I_* I^*[-2r]$, where $I\colon X \times \{0\} \to E$ is the inclusion of the zero section. Hence, taking cohomology, we need only show that $\cH^j I_* I^* {\rm Sp}(B_f^H) = 0$ for all $j > -r$. 

This vanishing is equivalent to the local complete intersection hypothesis. To see this, recall that one can characterize local complete intersections by the vanishing of lower local cohomology modules: $Z \subseteq X$ defined by $f_1,\dots, f_r$ is a complete intersection if and only if $\cH^j_Z(\cO_X) = 0$ for all $j < r$. These modules can be computed using $I_*I^! {\rm Sp}(\cB_f^H)$. By applying duality for $\cD$-modules, their vanishing can be determined from that of $\cH^j I_* I^* {\rm Sp}(\cB_f)$ for $j > -r$, using the fact that ${\rm Sp}$ commutes with duality and that $\cB_f$ is self dual.
\end{rmk}

By Kashiwara's equivalence (see Example \ref{KashiwaraEq}), the submodule $\Gamma_{X\times \{0\}}(\mu(\cB_f^H))$ can be described as $\cK[\partial_{y_1},\dots, \partial_{y_r}]$, where
\[ \cK = \bigcap_{i=1}^r \ker(y_i \colon {\rm Gr}_V^r \cB_f \to {\rm Gr}_V^{r-1} \cB_f),\]
where $\theta_y = \sum_{i=1}^r y_i\de_{y_i}$. Hence, the $\cD_E$-module underlying $L$ is
\[ \cL \coloneqq \cK[z_1,\dots, z_r], \text{ where }  \cK = \bigcap_{i=1}^r \ker(\partial_{t_i}\colon {\rm Gr}_V^r \cB_f \to {\rm Gr}_V^{r-1}\cB_f),\]
where $s= - \sum_{i=1}^r \de_{t_i} t_i$.

In particular, for any $\chi \notin \Z_{\geq r}$, the natural projection ${\rm Gr}_V^\chi \cB_f \to \cQ^\chi$ is an isomorphism.

We can give another, more geometric, description of the map $\Psi$. To do this, we will use the definition of the specialization given in Section \ref{sect-background}. Recall that we denote by $\widetilde{X}$ the deformation to the normal bundle of $X\times \A^r_t$ along $X\times \{0\}$.

By construction, we have the natural map $p\colon \widetilde{X} \to X \times \A^r_t$, and we can consider $p^*(\Gamma_* M) \in {\rm D}^b({\rm MHM}(\widetilde{X}))$. By factoring $q = p \circ j$, we have an isomorphism $j^* p^*(\Gamma_* M) = q^*(\Gamma_* M)$. Hence, by Equation \ref{sameNearbyCycles}, we conclude that
\[ \psi_u p^*(\Gamma_* M) = \psi_u j_* q^*(\Gamma_* M) = {\rm Sp}(\Gamma_* M).\]

The composition $\rho\colon X\times \A^r_z \xrightarrow[]{\iota} \widetilde{X} \xrightarrow[]{p} X\times \A^r_t$ is given by $(x,z_1,\dots, z_r) \mapsto (x,0,\dots, 0)$. Hence, we have a Cartesian diagram
\[ \begin{tikzcd} Z \times \A^r_z \ar[r,"i"] \ar[d] & X\times \A^r_z \ar[d,"\rho"] \\ X \ar[r,"\Gamma"] & X \times \A^r_t
\end{tikzcd}.\]

There is always a canonical natural transformation $\iota^* \to \psi_u$ in ${\rm D}^b({\rm MHM}(\widetilde{X}))$ (see \cite[(2.24.3)]{SaitoMHM}) from na\"{i}ve restriction to the nearby cycles functor. Looking at this morphism for $p^*(M)$, we get
\[ \rho^*(\Gamma_* M) \to {\rm Sp}(\Gamma_* M).\]

For $M = \Q_X^H[\dim X]$, we can use Saito's base-change result \cite[(4.4.3)]{SaitoMHM} to see that $\rho^*(B_f^H) = i_* \Q_{Z\times \A^r_z}^H[\dim X]$. When $Z$, hence $Z\times \A^r_z$, is a local complete intersection, by \cite[Thm. 5.1.20]{Dimca} the perverse sheaf underlying this mixed Hodge module is $\Q_{Z^{\rm an}\times \C^r}[\dim Z + r]$. Following the argument of \cite[(1.4.7)]{SaitoRationalPowers}, and using the fact that the functor ${\rm MHM}(-) \to {\rm Perv}(-)$ is faithful, we see that 
\begin{equation} \label{noQuotient} L = \rho^*(B_f^H) \text{ has no non-zero quotient objects which are supported on } Z_{\rm sing} \times \A^r \end{equation}

Using this we can see that the two methods of defining $\Psi$ agree, as they give the same morphism upon restricting to $V \times \A^r_z$, where $V \subseteq X$ is the open subset such that $V \cap Z = Z_{\rm reg}$. The morphism $\Psi\vert_{V\times \A^r_z}$ is actually an isomorphism, so the cokernel $Q$ is supported on $Z_{\rm sing} \times \A^r_z$.

We now compare our construction with what was done in the $r=1$ case. In \cite{SaitoMicrolocal}, in the case of a single function $h \in \cO_S$, Saito defines the \emph{algebraic partial microlocalization} as follows: let $\cB_h = \bigoplus_{k\in \N} \cO_S \de_\zeta^k \delta_h$ and set
\[ \widetilde{\cB}_h \coloneqq \bigoplus_{k \in \Z} \cO_S \de_{\zeta}^k \delta_h = \cB_h[\de_\zeta^{-1}].\]

The ring $\cD_{S \times \A^1_{\zeta}}[\de_\zeta^{-1}]$ admits a $\Z$-indexed $V$-filtration given by 
\[ V^\bullet \cD_{S\times \A^1_{\zeta}}[\de_\zeta^{-1}] = \sum_{k \in \Z} \de_\zeta^{-k} V^{\bullet-k} \cD_{S\times \A^1_{\zeta}}.\]

Saito defines the $V$-filtration on $\widetilde{\cB}_h$ by
\[ V^\gamma \widetilde{\cB}_h \coloneqq \begin{cases} V^\gamma \cB_h + \cO_S[\de_\zeta^{-1}]\de_{\zeta}^{-1}\delta_h & \gamma \leq 1\\ \de_\zeta^{-j} V^{\gamma-j}\widetilde{\cB}_h & \gamma > 1, 0 < \gamma - j \leq 1\end{cases},\]
which is compatible with the filtration $V^\bullet$ on $\cD_{S\times \A^1_{\zeta}}[\de_\zeta^{-1}]$.

The operator $s = -\de_{\zeta} \zeta$ satisfies $(s+\lambda)^N V^{\lambda}\widetilde{\cB}_h \subseteq V^{>\lambda}\widetilde{\cB}_h$ for some $N \gg 0$, using the corresponding property for the usual $V$-filtration. From this, we see that for all $\lambda \neq 0$, the natural map $\zeta: {\rm Gr}_V^{\lambda} \widetilde{\cB}_h \to {\rm Gr}_V^{\lambda+1} \widetilde{\cB}_h$ is an isomorphism.

If we define a filtration $F_\bullet \widetilde{\cB}_h$ by 
\[ F_p \widetilde{\cB}_h = \bigoplus_{k \leq p + \dim S} \cO_S \de_\zeta^k \delta_h,\]
then we can quickly see (as in \cite[(2.1.4)]{SaitoMicrolocal}) that the natural map $\cB_h \to \widetilde{\cB}_h$, which obviously preserves the $V$-filtration, satisfies
\[F_p {\rm Gr}_V^\gamma \cB_h \cong F_p {\rm Gr}_V^{\gamma}\widetilde{\cB}_h \text{ for all } \gamma < 1,\]
and so we see that the vanishing cycles
\begin{equation}\label{microlocalComputeVanCycles} \phi_h(\Q_S^H[\dim S]) \text{ has } \bigoplus_{\lambda \in [0,1)} ({\rm Gr}_V^{\lambda}\widetilde{\cB}_h,F[\lceil \lambda\rceil ]) \text{ as underlying filtered }\cD\text{-module,}\end{equation}
where for any $\ell \in \Z$, the shifted filtration $F[\ell]_\bullet$ is equal to $F_{\bullet-\ell}$.

\begin{lem} \label{comparer1} The kernel of the canonical map ${\rm Gr}_V^1 \cB_h \to {\rm Gr}_V^1 \widetilde{\cB}_h$ is equal to \[\cK = \ker(\de_\zeta\colon {\rm Gr}_V^1 \cB_h \to {\rm Gr}_V^0\cB_h).\] Hence the module $\cQ$ agrees with the image of the natural morphism $\bigoplus_\chi {\rm Gr}_V^\chi \cB_h \to \bigoplus_\chi {\rm Gr}_V^\chi \widetilde{\cB}_h$.
\end{lem}
\begin{proof} Let $u  \in V^1 \cB_h$ be an element which maps to $0$ in ${\rm Gr}_V^1 \widetilde{\cB}_h$. This means
\[u\in V^{>1}\widetilde{\cB}_h = \de_\zeta^{-1} V^{>0}\widetilde{\cB}_h,\]
which means
\[ \de_{\zeta} u = v + \epsilon \quad \text{with}\quad  v\in V^{>0}\cB_h \text{ and } \epsilon\in \cO_{Y\times \A^1}[\de_{\zeta}^{-1}]\de_\zeta^{-1}\delta_h.\]

As $\epsilon = \de_{\zeta} u - v\in \cB_h$, this forces $\epsilon = 0$. So $\de_{\zeta} u \in V^{>0}\cB_h$.

For the second claim, we have already mentioned the isomorphisms ${\rm Gr}_V^{\lambda}\cB_h \to {\rm Gr}_V^{\lambda}\widetilde{\cB}_h$ for $\lambda < 1$. From the commutative diagram with horizontal maps being isomorphisms for $\lambda \neq 0$,
\[ \begin{tikzcd} {\rm Gr}_V^\lambda \cB_h \ar[r,"\zeta"] \ar[d] & {\rm Gr}_V^{\lambda+1} \cB_h \ar[d]\\ {\rm Gr}_V^{\lambda} \widetilde{\cB}_h \ar[r,"\zeta"] & {\rm Gr}_V^{\lambda+1} \widetilde{\cB}_h\end{tikzcd}\]
we conclude inductively that for all $\lambda \notin \Z$, the map ${\rm Gr}_V^{\lambda}\cB_h \to {\rm Gr}_V^{\lambda}\widetilde{\cB}_h$ is an isomorphism.

By the first claim in this lemma, we see that the image of ${\rm Gr}_V^1 \cB_h \to {\rm Gr}_V^1 \widetilde{\cB}_h$ is isomorphic to $\cQ^1 = \cK$. The claim for $j\in \Z_{\geq 1}$ then follows again using multiplication by $\zeta$.
\end{proof}

Now, we let $Z = V(f_1,\dots, f_r) \subseteq X$ and consider $Y = X \times \A^r_y$ as in the definition of the minimal exponent of $Z$. We consider the function $g = \sum_{i=1}^r y_i f_i$.

We define subspaces $\widetilde{\cB}_g^{(k)} \coloneqq \bigoplus_{\alpha \in \N^r} \cO_X y^\alpha \de_{\zeta}^{|\alpha|-k} \delta_g \subseteq \widetilde{\cB}_g$, which give a direct sum decomposition 
\[ \widetilde{\cB}_g = \bigoplus_{k \in \Z} \widetilde{\cB}_g^{(k)}.\]

By \cite[(2.2.3)]{SaitoMicrolocal}, we have bifiltered isomorphisms
\[ \de_\zeta^k\colon F_p V^\lambda \widetilde{\cB}_g \to F_{p+k} V^{\lambda-k}\widetilde{\cB}_g,\]
which induce, for all $\ell,p\in \Z, \lambda \in \Q$, isomorphisms 
\begin{equation} \label{MicrolocalBifiltIso}\de_{\zeta}^k\colon F_p {\rm Gr}_V^{\lambda} \widetilde{\cB}_g^{(\ell)} \cong F_{p+k}{\rm Gr}_V^{\lambda-k} \widetilde{\cB}_g^{(\ell-k)}.\end{equation}

It is easy to see that $\widetilde{\cB}_g^{(k)}$ is the $k$-eigenspace of the operator $\theta_y - s$, where $\theta_y = \sum_{i=1}^r y_i\partial_{y_i}$ and $s = -\de_{\zeta}\zeta$. Hence, as this operator preserves $V^\gamma \widetilde{\cB}_g$ for all $\gamma \in \Q$, we see that these subspaces decompose into their eigenspaces. This is clearly true for $F_p V^\gamma \widetilde{\cB}_g$, too. In particular,
\[ F_p {\rm Gr}_V^\gamma \widetilde{\cB}_g  = \bigoplus_{k \in \Z} F_p {\rm Gr}_V^{\gamma} \widetilde{\cB}_g^{(k)}.\]

By equation \ref{microlocalComputeVanCycles} and definition of the Hodge filtration on nearby cycles, we have defined a filtered isomorphism
\begin{equation} \label{defineRho}\rho: (\phi_g(\cO_Y),F) \cong \bigoplus_{\lambda \in [0,1)}\bigoplus_{k\in \Z} ({\rm Gr}_V^{\lambda} \widetilde{\cB}_g^{(k)},F[\lceil \lambda\rceil]).\end{equation}

We recall the morphism $\varphi\colon \widetilde{\cB}_g \to \cB_f$ defined in \cite{CDMO}. It is the unique $\cO_X$-linear map such that
\[ \varphi(y^\alpha \de_\zeta^j \delta_g) = \de_t^\alpha \delta_f \text{ for all } \alpha \in \N^r, j\in \Z.\]

We will use a variant of this $\varphi$ which is defined to more closely match the $\cD_Y$-module structure of the Fourier-Laplace transform. In particular, in this paper, we define $\varphi\colon \widetilde{\cB}_g \to \cB_f$ as the unique $\cO_X$-linear map such that
\[ \varphi(y^\alpha \de_\zeta^j \delta_g) = (-1)^{j+|\alpha|} \de_t^\alpha \delta_f \text{ for all } \alpha \in \N^r, j\in \Z.\]

It is easy to see that $\varphi$ restricted to any summand $\widetilde{\cB}_g^{(k)}$ is bijective. We recall some other properties of the map $\varphi$, whose proofs are identical to those found in the proof of \cite[Prop. 3.2]{CDMO}:
\begin{lem} \label{PhiProperties} For $\varphi\colon \widetilde{\cB}_g \to \cB_f$ defined as above, the following hold:\hfill
\begin{enumerate} \item The map $\varphi$ is $\cD_X$-linear.

\item For all $u \in \widetilde{\cB}_g$, we have $\varphi(\de_\zeta u) = - \varphi(u)$.

\item For all $u \in \widetilde{\cB}_g$, we have $\varphi(y_i u) = -\de_{t_i} \varphi(u)$.

\item For all $u \in \widetilde{\cB}_g$, we have $\varphi(\de_{y_i} u) = t_i \varphi(u)$.

\item $($\cite[Thm 3.3]{CDMO}$)$ For every $\gamma \in \Q$ and $p,k \in \Z$, we have \[F_p V^\gamma \widetilde{\cB}^{(k)}_g = \varphi^{-1}(F_{p+r+k} V^{\gamma-k} \cB_f) \cap \widetilde{\cB}_g^{(k)}.\]
\end{enumerate}
\end{lem}

The last claim is not stated with the Hodge filtration in \cite{CDMO} and it is only stated there for $k = 0$. Using the action of $\de_\zeta^k$ and Equation \ref{MicrolocalBifiltIso}, the claim for arbitrary $k$ reduces to that for $k = 0$. The Hodge filtration is easy to handle, keeping in mind that we are indexing with the conventions of right $\cD$-modules.

\begin{prop} \label{justification} The map
\[\overline{\varphi}: \bigoplus_{\lambda \in [0,1)} \bigoplus_{k \in \Z} {\rm Gr}_V^{\lambda} \widetilde{\cB}_g^{(k)},F[\lceil \lambda\rceil]) \to (\mu(\cB_f),F[-r])\]
induced by $\varphi$ is an isomorphism of filtered $\cD_Y$-modules. Hence, the composition
\[ \overline{\varphi} \circ \rho: (\phi_g(\cO_Y),F) \to (\mu(\cB_f),F[-r])\]
is an isomorphism of filtered $\cD_Y$-modules.
\end{prop}
\begin{proof} By Property (5) of Lemma \ref{PhiProperties}, we see that the map $\varphi$ induces for any $\lambda \in [0,1)$ an isomorphism
\[ F_{p-\lceil \lambda\rceil }{\rm Gr}_V^{\lambda} \widetilde{\cB}_g^{(k)} \cong \begin{cases} F_{p+k+r}{\rm Gr}_V^{-k} \cB_f & \lambda = 0 \\  \bigoplus_{k \in \Z} F_{p-1+k+r}{\rm Gr}_V^{\lambda-k} \cB_f& \lambda \in (0,1) \end{cases}.\]

By taking the direct sum over $k\in \Z$ and using the formula \ref{HodgeFL} (noting that $\lceil \gamma \rceil = 1$ for $\gamma \in (0,1)$), we conclude that this is isomorphic to $F_{p+r} {\rm FL}{\rm Sp}(\cB_f) = F_{p+r}\mu(\cB_f)$.
\end{proof}

This gives some heuristic reason why the microlocalization $\mu(\cB_f)$ and the map $\Psi$ should help us understand $\widetilde{\alpha}(Z) = \widetilde{\alpha}(g\vert_U)$.

\section{Results on Spectrum} \label{sect-spectrum} We recall the notation from Section \ref{sect-background}. For $x\in Z$, we let $i_x\colon \{x\} \times \A^r_z \to X \times \A^r_z$ be the closed embedding, and for $\xi \in \{x\} \times \A^r_z$ we let $j_\xi\colon \{\xi\} \to \{x\}\times \A^r_z$ and $i_\xi \coloneqq i_x \circ j_\xi$ be the corresponding closed embeddings. The spectrum of a monodromic mixed Hodge module $M$ at $x$ is defined as
\[ \widehat{\rm Sp}(\cM,x) = \sum m_{\alpha,x}(\cM) t^\alpha,\]
where
\[ m_{\alpha,x}(\cM) \coloneqq \sum_{k\in \Z} (-1)^k \dim_\C {\rm Gr}^F_{\lceil \alpha \rceil -\dim Z -1} \cH^{k-r} i_\xi^* \cM^{\alpha + \Z}.\]

Recall that the \emph{reduced Spectrum} of $Z$ at $x$ is by definition
\[ {\rm Sp}(Z,x) = \widehat{\rm Sp}(Z,x) + (-t)^{\dim Z +1} = \sum \overline{m}_{\alpha,x} t^\alpha,\]
where we are defining $\overline{m}_{\alpha,x}$ by this formula.

\begin{lem} We have $\widehat{\rm Sp}(i_* \Q_{Z\times \A^r_z}^H[\dim X],x) = -(-t)^{\dim Z +1}$. Hence, 
\[ {\rm Sp}(Z,x) = \widehat{\rm Sp}(Q,x).\]
\end{lem}
\begin{proof} As $x\in Z$, the inclusion $i_x\colon \{x\} \times \A^r_z \to X \times \A^r_z$ factors through $i\colon Z \times \A^r_z \to X\times \A^r_z$, say $i_x = i\circ \iota_x$. Then, because $i^* i_* = \text{id}$, we have
\[ i_x^* (i_* \Q_{Z\times \A^r_z}^H[\dim X]) = \iota_x^* \Q_{Z \times \A^r_z}^H[\dim X] = \Q_{\{x\}\times \A^r}^H[\dim X],\]
where the last equality is by definition \ref{trivialHodgeModule} of $\Q_W^H$ on any complex algebraic variety $W$.

Applying $j_\xi^*$ to this, we get $\Q_{\{\xi\}}^H[\dim X]$, again by the definition. From here the claim reduces to checking indices, which is left to the reader.
\end{proof}

\begin{proof}[Proof of Theorem \ref{thm-Spectrum}] The proof is divided into two steps: the first is to show vanishing of the spectral numbers, which can be done without assuming isolated singularities.

The claim of Theorem \ref{thm-Spectrum} concerns the local minimal exponent $\widetilde{\alpha}_x(Z)$ for $x\in Z$. This is defined as $\widetilde{\alpha}(V\cap Z)$ for any small enough open neighborhood $x\in V \subseteq X$. However, the inclusion $i_x\colon \{x\} \times \A^r_z \to X\times \A^r_z$ clearly factors through $V \times \A^r_z$, so by replacing $X$ by $V$, we can work with the global minimal exponent $\widetilde{\alpha}(Z)$.

We will use Lemma \ref{lem-SpectralNumberLowerBound}, so we need to study $F_{p-\dim X} \cQ$.

Assume that $\widetilde{\alpha}_x(Z) > r-1$. If $\widetilde{\alpha}_x(Z) \in (r-1,r]$, then we have $F_{-\dim X} \cB_f \subseteq V^{\widetilde{\alpha}_x(Z)} \cB_f$.

If $\widetilde{\alpha}_x(Z) > r$, then by \cite[Thm 1.1]{CDMO}, we have for $\lambda \in (0,1], p \in \Z_{\geq 0}$ that
\[ \widetilde{\alpha}(Z) \geq r+p+\lambda \iff F_{p+1-\dim X} \cB_f \subseteq V^{r-1+\lambda}\cB_f.\]

To handle both cases at once, write $\widetilde{\alpha}_x(Z) = r-1+p+\lambda$ for some $p \in \Z_{\geq 0}$. Then we see that $F_{p-\dim X} \cB_f \subseteq V^{r-1+\lambda}\cB_f$ and it is not contained in $V^{>r-1+\lambda}\cB_f$. For any $\chi < r-1+\lambda$, the first containment shows
\[ F_{p-\dim X} {\rm Gr}_V^{\chi}(\cB_f) = 0,\]
and in particular, for any $\mu \in (0,\lambda)$, we have
\[ F_{p - \dim X} \cQ^{r-1+\mu} = F_{p-\dim X} {\rm Gr}_V^{r-1+\mu}(\cB_f) = 0,\]
where the first equality follows from $\mu \notin \Z$. Using the filtered surjection \ref{FilteredSurj}, we see that $F_{p-\dim X} \cQ^{\mu+\Z} = 0$ for $\mu \in (0,\lambda)$.

We see also that $F_{p-1-\dim X} \cQ = 0$. For $p = 0$, this is obvious. For $p > 0$, we have $\widetilde{\alpha}(Z) > r$, and so $F_{p-1-\dim X} \cB_f \subseteq V^r \cB_f$ and 
\[\de_{t_i} F_{p-1- \dim X} \cB_f \subseteq F_{p-\dim X} \cB_f \subseteq V^{r-1+\lambda}\cB_f \subseteq V^{>r-1}\cB_f\]
so that $F_{p-1- \dim X} {\rm Gr}_V^r(\cB_f) = F_{p-1+\dim X} \cK$, proving the vanishing.

By Lemma \ref{lem-SpectralNumberLowerBound}, we get $m_{\alpha,x}(Q) = 0$ for all $\alpha < p + \lambda = \widetilde{\alpha}_x(Z) - r +1$.

To show equality in the theorem statement, we assume $Z_{\rm sing} = \{x\}$.

We know that for any $\lambda \in (0,1]$, the module $\cQ^{\lambda + \Z}$ is supported on $Z_{\rm sing} \times \A^r_z = \{x\} \times \A^r_z$. We have in this case, by \cite[Prop. 2.1(ii)]{DMS} the equality
\[ F_{p-\dim X} i_x^* \cQ^{\lambda + \Z} = F_{p-\dim X} \cQ^{\lambda+\Z} \neq 0.\]

By the computation above, this lowest Hodge piece satisfies the hypotheses of Lemma \ref{lem-minHodgeGenSupp}. In particular, by the fact that $\xi$ is chosen generally so that $j_\xi$ is non-characteristic, we have that
\[ F_{p-\dim Z} i_\xi^* \cQ^{\lambda + \Z} \neq 0,\]
as this agrees with the $\cO$-module pull-back of $F_{p-\dim X} \cQ^{\lambda+\Z}$ to $\{\xi\}$, which is non-zero as ${\rm Supp}(F_{p-\dim X} \cQ^{\lambda+\Z}) = \A^r$. This finishes the proof.
\end{proof}

We proceed with the proof of the theorem on equisingular families of ICIS subvarieties.

\begin{proof}[Proof of Theorem \ref{thm-equising}]  The proof is essentially immediate from \cite[Theorem 2]{DMS}. We recall the set-up from the introduction: let $\cX \subseteq Y \times T$ be a subvariety which contains $\{0\} \times T$ (for some distinguished point $0\in Y$) and which is a complete intersection of pure codimension $r$. Assume that $\cX_t \coloneqq \cX \cap (Y \times \{t\})$ is a complete intersection subvariety of $Y$ of pure codimension $r$ with an isolated singularity at $0$.

Let $U = \cX \setminus (\{0\} \times T)$ and, as in the theorem statement, assume that $U$ is smooth and that the pair $(\{0\}\times T, U)$ satisfies Whitney's conditions along $\{0\} \times T$. Then we get a Whitney stratification of $N_{\cX} (Y\times T) = \cX \times \A^r_z$ simply by $(\{0\}\times T \times \A^r_z, U \times \A^r_z)$, so that the morphism $N_{\cX} (Y\times T) \to \cX$ is stratified. 

As $U$ is smooth, it is clear that ${\rm Sp}_{\cX}(Y\times T)$ is constant on $U\times \A^r_z$. The additional assumption that ${\rm Sp}_\cX(Y\times T)$ is locally constant along $\{0\} \times T$ ensures that this Whitney stratification satisfies the assumptions of \cite[Theorem 2]{DMS}. Hence, we know that ${\rm Sp}(\cX,(0,t))$ does not depend on $t\in T$. Moreover, by slicing with the transversal $Y\times \{t\}$, we see that this spectrum is equal to ${\rm Sp}(\cX_t,0)$, proving the claim.
\end{proof}

\section{Results on Bernstein-Sato roots} \label{sect-Bernstein}
In this section, $Z \subseteq X$ is defined by $f_1,\dots, f_r \in \cO_X(X)$, the morphism $\Gamma\colon X \to X\times \A^r_t$ is the graph embedding and $\cB_f = \Gamma_+ \cO_X$ the corresponding $\cD$-module on $X\times \A^r_t$.

We begin by proving Theorem \ref{thm-Bernstein}. For this, we take $f_1,\dots, f_r$ to be a regular sequence, so that $Z$ is a complete intersection.

First of all, the $G$-filtration on $\cB_f$ induces one on ${\rm Gr}_V^\gamma \cB_f$ for all $\gamma \in \Q$, hence, a filtration on ${\rm Sp}(\cB_f)^{\lambda + \Z}$ for any $\lambda \in (0,1]$, defined as
\[ G^\bullet {\rm Sp}(\cB_f)^{\lambda + \Z} \coloneqq \bigoplus_{\ell \in \Z} G^{\bullet + \ell} {\rm Gr}_V^{\lambda+r-1+\ell} \cB_f,\]
which is indexed so that this is a filtration by $\cD$-submodules.

It induces a filtration $G^\bullet \cK[z_1,\dots,z_r]$ on the submodule $\cL  = \cK[z_1,\dots, z_r]$.

\begin{lem} \label{GFiltonK} We have $G^0 \cL = \cL$. In particular, for all $i > 0$, we have
\[ {\rm gr}_G^{-i} \cL = 0.\]
\end{lem}
\begin{proof} Recall from Equation \ref{noQuotient} that $\cK[z_1,\dots,z_r]$ has no non-zero quotients which are supported on a proper closed subset of $Z\times \A^r_z$. Hence, the result follows by restricting to $V \subseteq X$ such that $V \cap Z = Z_{\rm reg}$, because it is easy to check that if $Z$ is smooth, then $G^0 {\rm Sp}(\cB_f) = {\rm Sp}(\cB_f)$. Thus, if we look at the composition $\cL \to {\rm Sp}(\cB_f) \to {\rm Sp}(\cB_f)/G^0 {\rm Sp}(\cB_f)$, the image vanishes upon restriction to $V$, hence it is 0.
\end{proof}

Let $\widetilde{\gamma}(Z) \coloneqq \min \{\gamma \mid \widetilde{b}_f(-\gamma) = 0\}$. By \cite[Theorem 1.6]{CDMO}, we have
\[ \widetilde{\gamma}(Z) \leq \widetilde{\alpha}(Z) \text{ and } \min\{r+1,\widetilde{\gamma}(Z)\} = \min\{r+1,\widetilde{\alpha}(Z)\},\]
so the result of the theorem is already known if either $\widetilde{\alpha}(Z)$ or $\widetilde{\gamma}(Z)$ is strictly less than $r+1$.

Thus, we assume $\widetilde{\gamma}(Z) \geq r+1$. The next lemma allows us to assume $\widetilde{\gamma}(Z)$ is an integer:
\begin{lem} Let $Y = X\times \A^r_y$, $U = Y \setminus (X\times \{0\})$ and $g = \sum_{i=1}^r y_i f_i$. Then there exists a finite set $I \subseteq \Z_{\geq 0}$ such that
\[ b_g(s) = b_{g\vert U}(s) \prod_{i\in I} (s+r+i).\]
\end{lem}
\begin{proof} Let $p(s) = b_{g\vert U}(s)$, so that, by definition, $p(s) {\rm gr}_G^0 \cB_{g\vert_U} = 0$. Note that $\cB_{g\vert U} = (\cB_g)\vert_U$ and the same is true of the $G$-filtration.

Thus, we see that the submodule $p(s) {\rm gr}_G^0 \cB_g \subseteq {\rm gr}_G^0 \cB_g$, when restricted to $U$, is 0. Thus, it is supported on $X \times \{0\}$. 

Now, let $\delta_g$ be the distinguished generator of $\cB_g$. It is easy to see that $\theta_y \delta_g = s \delta_g$, where $s = -\de_{\zeta} \zeta$. Moreover, by definition of the $G$-filtration, $V^0\cD \cdot \delta_g = G^0\cB_g$, so ${\rm gr}_G^0 \cB_g$ is generated by the class of $\delta_g$ over ${\rm Gr}_V^0\cD = \cD_Y[s]$. Hence, the minimal polynomial of the $s$ action on $\delta_g$ agrees with $b_g(s)$. By what we have just observed, this is also equal to the minimal polynomial of the $\theta_y$ action on $\delta_g$.

As $\theta_y$ commutes with $s$, we see that $\theta_y p(s) \delta_g = p(s) \theta_y \delta_g = s p(s)\delta_g$, and $p(s)\delta_g$ generates $p(s) {\rm gr}_G^0\cB_g$. Hence, the minimal polynomial of the $s$ action on $p(s) {\rm gr}_G^0 \cB_g$ is the same as the minimal polynomial of the $\theta_y$-action on $p(s)\delta_g$.

Finally, Kashiwara's equivalence \ref{KashiwaraEq} shows that the eigenvalues of $\theta_y$ on a module supported on $\{y_1 = \dots = y_r =0\}$ lie in $\Z_{\leq -r}$. and if $u$ is a generalized eigenvector, it is actually an eigenvector. Putting this together, we see that the minimal polynomial $q(w)$ of the $\theta_y$-action on $p(s)\delta_g$ in $p(s) {\rm gr}_G^0 \cB_g$ is of the form
\[ q(w) = \prod_{i\in I} (w+r+i)\]
for some finite set $I \in \Z_{\geq 0}$.
\end{proof}

By dividing both sides in the Lemma statement by $(s+1)$, we get equality (using \cite[Thm. 1.1]{Mustata21})
\[ b_f(s) = \widetilde{b}_{g\vert U}(s) \prod_{i\in I} (s+r+i),\]
and so we see that if $\widetilde{\gamma}(Z) \notin \Z$, then it must agree with $\widetilde{\alpha}(Z)$. So we can take $\widetilde{\gamma}(Z) = r + j$ for some $j\in \Z_{\geq 1}$.

\begin{proof}[Finishing the proof of Theorem \ref{thm-Bernstein}] Assume toward contradiction that $\widetilde{\alpha}(Z) > \widetilde{\gamma}(Z) = r+j$. The assumption on $\widetilde{\alpha}(Z)$ tells us that for all $|\beta| = j$, we have
\[ \de_t^{\beta}\delta_f \in V^r \cB_f \text{ and } \partial_{t_i} \de_t^\beta \delta_f \in V^{>r-1}\cB_f\]
for all $1\leq i\leq r$. Hence, the class $[\de_t^\beta \delta_f] \in {\rm Gr}_V^r \cB_f$ lies in $\cK$ for all $|\beta| = j$.

For $\ell \geq 0$, we have $G^{-\ell } \cB_f = \sum_{|\gamma| \leq \ell} V^0 \cD_{X\times \A^r_t} \cdot \partial_t^\gamma \delta_f$, by definition. Hence, we see that, for all $0 \leq \ell \leq j$, the subspace
\[ G^{-\ell} {\rm Gr}_V^r \cB_f \subseteq {\rm Gr}_V^r \cB_f\]
is generated over ${\rm Gr}_V^0 \cD_{X\times \A^r_t}$ by the classes $[\partial_t^\gamma \delta_f]$ for $|\gamma| \leq \ell$.

It is easy to check that $\cK \subseteq {\rm Gr}_V^r \cB_f$ is a ${\rm Gr}_V^0\cD_{X\times \A^r_t}$-submodule, and so putting this together we conclude that
\[ G^{-\ell} \cK = G^{-\ell} {\rm Gr}_V^{r} \cB_f,\]
for all $0\leq \ell \leq j$, and thus, by iterating the surjections \ref{KoszulSurj}, we conclude by definition of the $G$-filtrations that
\[ G^{-\ell} \cK[z_1,\dots, z_r] = G^{-\ell} {\rm Sp}(\cB_f) \text{ for } 0 \leq \ell \leq j.\]

But by Lemma \ref{GFiltonK}, this shows that ${\rm gr}_G^{-j} {\rm Sp}(\cB_f) = 0$. By looking at the $r+j$th monodromic piece, we conclude that
\[ {\rm gr}_G^0 {\rm Gr}_V^{r+j} \cB_f =0,\]
and so $(s+r+j) \nmid b_f(s)$ by the equivalence \ref{GFiltRoots}.
\end{proof}

\begin{proof}[Proof of Theorem \ref{thm-Construct}] As ${\rm Sp}(\cB_f)\vert_{S_j}$ is locally constant, the same is true for all sub-quotients. In particular, it is true for
\[ \cN_\gamma : =  \cD_{X\times \A^r_t} \cdot {\rm gr}_G^0 {\rm Gr}_V^\gamma \cB_f,\]
and, if we define the $\cD$-linear map $N: {\rm Sp}(\cB_f) \to {\rm Sp}(\cB_f)$ by having it act on the $\chi$th monodromic piece by $s+\chi$, we see that $N$ preserves the subspaces $\cN_\gamma$ for any $\gamma \in \Q$ and
\[ N^k \cN_\gamma \text{ is locally constant on } S_j \text{ for all } k \geq 0.\]

The claim follows immediately from this observation and the fact that $N^k \cN_\gamma \neq 0$ if and only if $(s+\gamma)^k {\rm gr}_G^0 {\rm Gr}_V^\gamma \cB_f \neq 0$ if and only if $(s+\gamma)^{k+1} \mid b_f(s)$, using equivalence \ref{GFiltRoots}.
\end{proof}

\end{document}